\documentclass[11pt]{amsart}
\usepackage{amsmath,amsfonts,amssymb,amscd,amsthm,amsbsy,epsf}
\newtheorem{thm}{Theorem}[section]
\newtheorem{lem}[thm]{Lemma}
\newtheorem{prop}[thm]{Proposition}
\newtheorem{adden}[thm]{Addendum}
\newtheorem{cor}[thm]{Corollary}

\newtheorem{quest}[thm]{Question}
\theoremstyle{definition}
\newtheorem{defn}[thm]{Definition}

\newtheorem{rem}[thm]{Remark}
\def\bA{{\mathbf A}}
\def\bB{{\mathbf B}}
\def\bD{{\mathbf D}}
\def\bL{{\mathbf L}}
\def\A{{\mathcal A}}
\def\B{{\mathcal B}}
\def\C{{\mathcal C}}
\def\D{{\mathcal D}}
\def\G{{\mathcal G}}
\def\K{{\mathcal K}}
\def\L{{\mathcal L}}
\def\M{{\mathcal M}}
\def\P{{\mathcal P}}
\def\S{{\mathcal S}}
\def\U{{\mathcal U}}
\def\Y{{\mathcal Y}}
\def\nat{{\mathbb N}}
\def\que{{\mathbb Q}}
\def\zed{{\mathbb Z}}
\def\Rec{\operatorname{Rec}}

\def\wtilde{\widetilde}
\def\too{\longrightarrow}
\def\LL{\langle\!\langle}
\def\GG{\rangle\!\rangle}
\begin{document}
\title{Unsolvable problems about higher-dimensional knots and 
related groups}
\author{F. Gonz\'alez-Acu\~na, C. M{\tiny c}A. Gordon and J. Simon}
\dedicatory{Dedicated to the memory of Michel Kervaire}
\maketitle
\markboth{}{}

\section{Introduction}

We consider in this paper the classes of groups $\K_0$, $\K_1$, $\K_2$, 
$\K_3$, $\S$, $\M$ and $\G$, each properly containing the preceding one, 
related to codimension~2 smooth embeddings of manifolds. 
$\K_n$ is the class of groups of complements of $n$-spheres in 
$S^{n+2}$; $\S$ (resp. $\M,\G$) is the class of groups of complements 
of orientable, closed surfaces in $S^4$ (resp. a 1-connected 4-manifold, 
a 4-manifold). 
$\G$ is in fact the class of all finitely presented groups and $\K_0$ 
contains only the infinite cyclic group.
We are interested in the problem of recognizing when a group in one
of these classes belongs to a smaller class.
In general, this is an unsolvable problem.

\begin{thm} \label{thm1.1}
Let $\A$ and $\B$ be members of $\{\K_0,\K_1,\K_2,\K_3,\S,\M,\G\}$ such 
that $\B\subsetneq \A$ and $\A\supset \K_3$. 
Then there does not exist an algorithm that can decide, given a finite 
presentation of a group $G$ in $\A$ whether or not $G$ is in $\B$.
\end{thm} 

For $\A=\K_1$, $\B = \K_0$ one can prove, using Haken's theorem
\cite{Ha}, that such an algorithm exists.  

We conjecture that Theorem~\ref{thm1.1} also holds 
for $\A=\K_2$;  we show that this is true if there is a group 
in $\K_2$ with unsolvable word problem.

The case $\A = \K_3$ and $\B = \K_0$ of Theorem~\ref{thm1.1}  implies 
that the isomorphism problem for $\K_3$ and the other three classes 
containing it is unsolvable. 
Conjecturally this should hold also for $\K_2$.

Though we do not know whether 2-knot groups with unsolvable word problem 
exist, we prove (Corollary~\ref{cor:wordprob}) that there are 3-knot groups with 
unsolvable word problem. 
This is a consequence of Theorem~\ref{thm7} which states that every 
finitely presented group embeds in a 3-knot group.

We use often methods of \cite{Gor2} in which, in fact, the case $(\A,\B)= 
(\G,\K_3)$ of Theorem~\ref{thm1.1}  is proved. 
The cases $(\A,\B) = (\G,\K_i)$, $i=0,1$, actually follow from 
\cite[Theorem~1.1]{R1}  
and the fact that the groups of $\K_1$ are torsion free \cite{P}; 
they were known to Baumslag and Fox (see \cite{St}).

It is also proved in \cite{Gor2} that the problem of deciding if the second 
homology of a finitely presented group $G$ is trivial is unsolvable.
The case $(\A,\B) = (\S,\K_3)$ of our theorem actually states that one 
cannot decide if a group in $\S$ has trivial second homology.

We show (Theorem~\ref{thm11}) that, in general, problems concerning 
the computation of the integral homology of finitely presented groups are 
unsolvable. 
We also prove (Theorems~\ref{thm12} and \ref{thm13}) incomputability 
results about the Whitehead groups $Wh_0(G)$ and $Wh_1(G)$,  
and Wall's surgery groups $L_n(G)$.

In the last two sections we prove a geometric unsolvability result: 
If $\K_n$ contains a group with unsolvable word problem then there is no 
algorithm which decides whether or not an $n$-sphere in $S^{n+2}$ is unknotted.
Since as we mentioned above, $\K_n$ contains groups with unsolvable word 
problem if $n\ge3$, it follows that no algorithm to decide if $n$-knots 
are trivial exists if $n\ge3$.
This result has been proved by Nabutovsky and Weinberger \cite{NW}. 
In contrast, Haken's classical result \cite{Ha} asserts that if $n=1$ 
such an algorithm exists. 

In Section~2 we define the various classes of knots being considered and 
give characterizations of the corresponding classes of groups. 
In Section~3 we give a particular way of effectively 
embedding an arbitrary group in a 
perfect group which will be useful in subsequent constructions. 
We then prove that some 3-knot groups are universal, that is, contain 
copies of every finitely presented group and, therefore, have unsolvable word
problem.
Also in Section~3 we prove Theorem~\ref{thm1.1} except for the case
$(\A,\B) = (\G,\M)$. 
We do this by using what we call an $(\A,\B,\C)$-construction to show 
that the solvability of the problem in question would imply the 
decidability of the triviality problem for finitely presented groups. 
In Section~4 we do the remaining case $(\A,\B) = (\G,\M)$; here the 
proof is based on the existence of finitely presented groups with unsolvable
word problem. 
In Section~5 we show that problems dealing with the homology, Whitehead 
groups and surgery groups of finitely presented groups are, in general, 
unsolvable. 
In Section~6 we give a recursive enumeration of $n$-knots; this is used in 
Section~7 where we derive the undecidability of the knotting problem for 
$n$-spheres from the existence of groups in $\K_n$ with unsolvable word 
problem.
As we mentioned above we do not know if such groups exist in $\K_2$.

The first-named author would like to thank Jon Simon, Dennis Roseman and 
the University of Iowa for their kind hospitality when he was a Visiting 
Professor; part of the present work was developed there. 

\section{Classes of knot groups}

In this section we define the classes of groups we are interested in.
We will be working in the PL category, 
and all embeddings will be locally flat. 
An {\em $n$-knot\/} is an $n$-sphere $\Sigma^n$ embedded in the $(n+2)$-sphere
$S^{n+2}$; the fundamental group of its complement, $\pi_1(S^{n+1}-\Sigma^n)$,
is called the {\em group\/} of the $n$-knot.

Two $n$-knots $(S^{n+2},\Sigma_1^n)$, $(S^{n+2},\Sigma_2^n)$ are 
{\em equivalent\/} if there is a PL homeomorphism from $S^{n+2}$ to 
$S^{n+2}$ mapping $\Sigma_1^n$ onto $\Sigma_2^n$. 
An {\em $n$-knot type\/} is an equivalence class of $n$-knots.

An $n$-knot $(S^{n+2},\Sigma^n)$ is {\em trivial\/} if there is an 
$(n+1)$-disk $D^{n+1}$ in $S^{n+2}$ such that $\partial D^{n+1} =\Sigma^n$.

For $n\ge 0$ we define $\K_n$ to be the class of groups of $n$-knots. 
It is well known 
(see \cite{Ar2}, 
\cite{Hi}, 
\cite{Fa}
\cite{Fo}, 
\cite{Ke1}, 
\cite{Z}) 
that $\{\zed\} = \K_0 \subsetneq \K_1 \subsetneq \K_2 \subsetneq \K_3 =\K_n$ 
for $n>3$.

Define $\S_n$ (resp.\ $\M_n$) to be the class of fundamental groups of 
complements of closed orientable $n$-manifolds embedded in $S^{n+2}$ 
(resp.\ a 1-connected $(n+2)$-manifold). 
One has $\K_1 = \S_1=\M_1$ by the 3-dimensional Poincar\'e Conjecture. 
Also, if $n\ge2$, $\S_2 = \S_n$ (see \cite{Si}) and $\M_2 = \M_n$, so we 
set $\S = \S_2$ and $\M = \M_2$. 
In fact $\M$ is the class of groups of complements of a 2-sphere embedded 
in a manifold of the form $S^2 \times S^2 \#\ldots\# S^2 \times S^2$
\cite{Gon1}.
Let $\G$ be the class of all finitely presented groups.

Kervaire \cite{Ke1} has given the following ``intrinsic'' (i.e. not 
involving presentations) group-theoretic characterization of $\K_3$.
The symbol $\LL t\GG$ denotes the normal closure of $t$.

\begin{thm}[Kervaire]\label{thm:Kervaire}
$\K_3 = \{G\in \G : H_1 (G) \cong \zed$, $H_2 (G) = 0$ and there exists 
$t\in G$ such that $\LL t\GG = G\}$.
\end{thm}

Also it is easy to see that 
$$\M = \{G\in \G:\ \text{\em there exists } t\in G
\ \text{\em such that } \LL t\GG = G\}\ .$$

We have 
$$\K_3 \subsetneq \S \subsetneq \M \subsetneq \G\ .$$
The fact that the inclusions $\S\subset \M$ and $\M\subset \G$ are proper
is obvious. 
The existence of groups $G\in \S$ with $H_2(G) \ne 0$ 
\cite{BMS}, \cite{Gor1}, \cite{Li}, \cite{M} 
shows that the inclusion $\K_3\subset \S$ is also proper. 

Before giving a group-theoretic characterization of the class $\S$ we recall  
the definition of the Pontrjagin product of two commuting elements of a group.
Suppose $a,b\in G$ and $[a,b] =1$. 
Then the {\em Pontrjagin product}
 of $a$ and $b$, which we denote by $a\wedge b$, 
is the image of the canonical generator of $H_2(\zed\times\zed)$ under 
$H_2(\zed\times\zed) \xrightarrow{(\varphi_{a,b})_*} H_2(G)$ where 
$\varphi_{a,b} : \zed\times\zed\to G$ is the homomorphism such that 
$\varphi_{a,b} (1,0) = a$ and $\varphi_{a,b} (0,1) =b$. 
If $t\in G$ and $C_t$ is the centralizer of $t$ in $G$, then we write 
$t\wedge C_t = \{t\wedge c: c\in C_t\}$. 

Notice that if $C_t$ is cyclic then $t\wedge C_t =0$ because 
$(\varphi_{t,c})_*$ factors through the trivial group $H_2(C_t)$.  

The following characterization of the groups in $\S$ is a slight 
reformulation of a theorem of Simon \cite{Si}, using a remark in \cite{BT}. 

\begin{thm}[Simon]  \label{thm2}
$\S = \{G \in \G :H_1 (G) \cong \zed$ and there exists $t\in G$ such that 
$\LL t\GG =G$ and $t\wedge C_t = H_2(G)\}$.
\end{thm}

We now give characterizations using presentations.

A {\em Wirtinger presentation\/} is a finite presentation $\langle x_1,\ldots,
x_m :r_1,\ldots,r_n\rangle$ such that each relator $r_k$ is of the form 
$x_i^{-1} w^{-1} x_j w$. 
Then (see \cite{Si})

\begin{thm}\label{thm3}
$\S = \{ G\in \G : H_1 G\cong \zed$ and $G$ has a Wirtinger 
presentation$\}$.
\end{thm}

Artin gave in \cite{Ar1} a characterization of 1-knot groups using 
presentations. 

\begin{thm}[Artin]\label{thm4} 
A group belongs to $\K_1$ if and only if it has a presentation 
$\langle x_1,\ldots,x_n : x_j^{-1} \beta_j$, $1\le j\le n\rangle$ such that 
\begin{itemize}
\item[(1)] For $j=1,\ldots,n$, $\beta_j$ is conjugate to $x_{\mu (j)}$ in 
the free group $F$ generated by $x_1,\ldots,x_n$,
\item[(2)] $\prod_{j=1}^n \beta_j  = \prod_{j=1}^n x_j$ 
in $F$, and 
\item[(3)] $\mu$ is the permutation $(1\ 2\cdots n)$.
\end{itemize}
\end{thm}

There are also characterizations of 2-knot groups (see \cite{Gon2} and 
\cite{Ka}): 

\begin{thm}[Gonz\'alez-Acu\~na]\label{thm5}
A group belongs to $\K_2$ if and only if it has a presentation of the form
$$\langle x_1,\ldots, x_n : x_{2i-1}^{-1} x_{2i},\quad x_j^{-1}\beta_j,\quad 
1\le i\le h,\quad 1\le j\le n\rangle$$
satisfying {\rm (1)} and {\rm (2)} above and also 
\begin{itemize}
\item[(3)] The permutations $\mu$ and $\prod_{i=1}^h (2i-1, 2i)$ 
generate a transitive group of permutations of $\{1,2,\ldots,n\}$
\item[(4)] $\langle x_1,\ldots,x_n : x_j^{-1} \beta_j$, $1\le j\le n\rangle$
and $\langle x_1,\ldots, x_n : x_j^{-1} \beta'_j$, $1\le j\le n\rangle$ 
present free groups where 
$$\beta'_j = 
\begin{cases} 
x_{j+1} \beta_{j+1} x_{j+1}^{-1}&\text{if $j$ is odd and $j<2h$}\\
\beta_{j-1}&\text{if $j$ is even and $j\le 2h$}\\
\beta_j&\text{if $j> 2h$}
\end{cases}
$$
\end{itemize}
\end{thm}

We recall that a set $S$ is {\em recursively enumerable\/} if there 
is an algorithm (effective procedure) that lists the elements of $S$.
For example it is clear that the set of all finite presentations of 
groups is recursively enumerable. 
If $S$ is recursively enumerable, a subset $R\subset S$ is {\em recursive\/} 
if both $R$ and $S\setminus R$ are recursively enumerable; equivalently,
there is an algorithm to decide whether or not a given element of $S$ 
belongs to $R$. 
Clearly the set of presentations in Theorem~\ref{thm4} is a recursive 
subset of the set of finite presentations and, as explained in \cite{Gon2}, 
so is the set of presentations in Theorem~\ref{thm5}.
\bigskip

If $\A \subset\G$, let $P(\A)$ denote the set of all finite presentations 
of members of $\A$. 
In order for the decision problem for $\B \subsetneq \A$ in 
Theorem~\ref{thm1.1} to be well-posed, it is necessary that the 
corresponding set of presentations $P(\A)$ be recursively enumerable. 
We now show that if $\A$ is $\K_0,\K_1,\K_2,\K_3, \S$ or $\M$ then 
$P(\A)$ is recursively enumerable.

Let $\P$ be the finite presentation $\langle x_1,\ldots,x_m : r_1,\ldots,
r_n\rangle$. 
An {\it identity\/} in $\P$ is a $t$-tuple $\pi = (p_1,\ldots,p_t)$ 
where each $p_i$ is a conjugate, in the free group $F$ on $x_1,\ldots,x_m$, 
of an element of $\{r_1,r_1^{-1},\ldots,r_n,r_n^{-1}\}$ and $p_1\cdots p_t=1$
in $F$.

If $K_{\P}$ is the standard 2-complex associated to $\P$ (well-defined up 
to homotopy \hbox{equivalence}) and $\pi$ is an identity in $\P$, then there is 
an associated map $f$ of an oriented 2-sphere $S^2$ into $K_{\P}$
(for details see \cite[pages 157, 150 and 151]{LS}); 
denote by $[\pi]$ the image under ${f_* :H_2(S^2)\to H_2(K_{\P})}$ of the 
canonical generator $H_2 (S^2)$. 
If $\pi_1,\ldots,\pi_s$ are finitely many identities in 
$\langle x_1,\ldots, x_m :r_1,\ldots,r_n\rangle$ we will say that 
$\langle x_1,\ldots,x_m : r_1,\ldots, r_n ; \pi_1,\ldots,\pi_s\rangle$ 
is a {\em presentation with identities\/}. 
If $\P$ is a presentation, $|\P|$ will denote the group presented by $\P$.

\begin{lem}\label{lem1} 
Let $\Re$ be a recursively enumerable set of finite presentations. 
Then $\{\P \in \Re :H_2 (|\P|) =0\}$ is recursively enumerable.
\end{lem}

\begin{proof} 
There is a recursive enumeration $\L$ of all the presentations with 
identities $\langle x_1,\ldots,x_m : r_1,\ldots, r_n; \pi_1,\ldots,\pi_s
\rangle$ such that $\P \in \Re$ and $[\pi_1],\ldots,[\pi_s]$ generate 
$H_2(K_{\P})$ where $\P = \langle x_1,\ldots,x_m :r_1,\ldots,r_n\rangle$.

Notice that if $\langle x_1,\ldots,x_m :r_1,\ldots,r_n; \pi_1,\ldots,\pi_s
\rangle$ is in $\L$ and $\P = \langle x_1,\ldots, x_m :r_1,\ldots, r_n
\rangle$ then $H_2 (|\P|) =0$ since every element of $H_2 (K_{\P})$ is 
spherical. 
Conversely if $H_2 (|\P|) =0$ where $\P = \langle x_1,\ldots,x_m: r_1,\ldots,
r_n\rangle$ then $\langle x_1,\ldots,x_m :r_1,\ldots,r_n; \pi_1,\ldots ,
\pi_s\rangle$ is in $\L$ for some choice of identities $\pi_1,\ldots,\pi_s$ 
in $\P$.

Hence, if we strike out the identities in $\L$ and eliminate repetitions 
we obtain a list of all the finite presentations $\P\in\Re$ such that 
$H_2 (|\P|) =0$.
\end{proof}

\begin{lem}\label{lem2} 
Let $\Re$ be a recursively enumerable set of finite presentations. 
Let $\wtilde \Re$ be the set of finite presentations $\wtilde\P$ such that 
$|\wtilde \P| \cong |\P|$ for some $\P \in \Re$. 
Then $\wtilde\Re$ is recursively enumerable.
\end{lem}

\begin{proof} 
Let $\P_1,\P_2,\P_3,\cdots $ be a recursive enumeration of the elements 
of $\Re$.

Using Tietze's Theorem one can give, for any $i$, a recursive enumeration 
$\P_{i1}, \P_{i2}, \P_{i3},\cdots$ of all finite presentations defining 
the same group as $\P_i$. 
Hence, from $\P_{ij}$, $i,j\in\nat$ one obtains a recursive enumeration 
of $\wtilde \Re$.
\end{proof}

We use the notation  of Lemma~\ref{lem2} in the proof of the following 
theorem.

\begin{thm}\label{thm6} 
Let $\A$ be one of the classes $\K_0,\K_1,\K_2,\K_3,\S,\M$. 
Then $P(\A)$ is recursively enumerable.
\end{thm}

\begin{proof}
\begin{itemize}
\item[(1)] Let $\A = \K_0$.
Take $\Re = \{\langle x:\ \rangle\}$ in Lemma~\ref{lem2}.

\item[(2)] $\A = \K_1$. By \cite[Satz 10]{Ar1} 
(see, for example, \cite[Theorem 9.2.2]{N})
there is a recursive set $\Re$ of finite presentations such that 
$\wtilde \Re$ is the set of presentations of members of $\K_1$.

\item[(3)] $\A=\K_2$. Use the same argument appealing to \cite{Gon2} 
instead of \cite{Ar1}.

\item[(4)] $\A= \S$. Again, use the same argument taking $\Re$ to be the set 
of Wirtinger presentations.

\item[(5)] $\A=\M$. Using Tietze's Theorem enumerate recursively all finite 
presentations of the trivial group with a positive number of generators. 
Deleting the first relator from each presentation 
in this list we obtain a list, with repetitions,
of all the presentations of members of $\M$ with a positive number of 
generators.

\item[(6)] $\A=\K_3$. Take a recursive enumeration $\Re$ of the presentations
$\P$ of members of $\S$ and apply Lemma~\ref{lem1}.
\end{itemize}
\end{proof}

If $\B \subset \A$ ($\subset \G$), and $P(\A)$ is recursively enumerable, 
we say that the {\em recognition problem 
$\Rec (\A,\B)$ is solvable\/} if there exists an algorithm
which decides, given a finite presentation of a group $G\in \A$, whether 
or not $G\in \B$; otherwise, {\em unsolvable\/}. 
Clearly if $\C \subset \B\subset \A$, with $P(\A)$ and $P(\B)$ 
recursively enumerable,  then $\Rec (\B,\C)$ unsolvable implies 
$\Rec (\A,\C)$ unsolvable.
The fact that $\Rec (\G,\{1\})$ is unsolvable underlies many of our results.

\section{Effective embedding theorems\\ and the unsolvability of some 
recognition problems}

In this section we prove Theorem~\ref{thm1.1} except for the case 
$(\A,\B) = (\G,\M)$. 
The proofs will make use of the following proposition.

\begin{prop}\label{prop1}
There is a computable function which takes an arbitrary finite presentation 
of a group $G$ and produces a finite presentation of a group $P$ such that 
\begin{itemize}
\item[(1)] $G$ embeds in $P$;
\item[(2)] $P$ is perfect, i.e., $H_1 (P)=0$;
\item[(3)] if $G=1$, then $P=1$.
\end{itemize}
\end{prop}

\begin{adden}\label{adden1} 
In Proposition~\ref{prop1}, we may assume in addition that 
\begin{itemize}
\item[(4)] if $G\ne1$ then $H_2(P)$ is infinite.
\end{itemize}
\end{adden}

\begin{proof}[Proof of Proposition~\ref{prop1}] 
Suppose we have a finite presentation for $G$ with $m$ generators, 
$x_1,\ldots,x_m$. 
Adjoin new generators $a,\alpha,b,\beta$ to form the iterated free product 
$(G * \langle a,\alpha\rangle) * \langle b,\beta\rangle$ of $G$ with two 
free groups of rank~2.  
Now add $m+4$ additional relations, as follows, to obtain $P$ (compare 
\cite[Proof of Theorem~3]{Gor2}).
\begin{itemize}
\item[(i)] $a\alpha a^{-1} = b^2$
\item[(ii)] $\alpha a\alpha^{-1} = b\beta b^{-1}$ 
\item[(iii)] $a^{2i} x_i \alpha^{2i} = \beta^{2i+2} b\beta^{-2i-2}$\quad 
$(i=1,\ldots,m)$
\item[(iv)] $[x_1,a] \cdots [x_m,a] = \beta^2 b \beta^{-2}$
\item[(v)] $[x_1,\alpha] \cdots [x_m,\alpha] = \beta b\beta b^{-1}\beta^{-1}$.
\end{itemize}
We can see from relations (iv) and (v) that abelianizing $P$ gives 
$b=\beta =0$; then (i) and (ii) imply $a=\alpha=0$, so by (iii), each $x_i=0$.
Thus $P$ is perfect. 
If $G=1$ then (iv) and (v) imply $b=\beta=1$, so, as above, we conclude 
$a=\alpha=1$ as well. 
To show that the natural map from $G$ to $P$ is an embedding, we claim that
when $G\ne 1$, $P$ is an amalgamated free product $(G*\langle a,\alpha\rangle)
*_E\langle b,\beta\rangle$, where $E$ is a free group of rank $m+4$. 
One can check that the words in $b$ and $\beta$ on the right side of 
equations (i)--(v) freely generate a subgroup $E$ of 
$\langle b,\beta\rangle$, and that the elements on the left are a basis for a
free subgroup of $G * \langle a,\alpha\rangle$.
To verify this, one shows that any product corresponding to a freely 
reduced non-trivial word in those elements represents a non-trivial 
element in $\langle b,\beta\rangle$ or the free product $G * \langle a,
\alpha\rangle$, respectively, by showing that it has positive length when 
expressed in normal form \cite[p.187]{LS}. 
We suppress the details, but have chosen the elements such that the 
possibilities of cancellation are sufficiently restricted that these 
may readily be supplied. 
It should be noted that the possibility that several $x_i=1$ does not 
make the elements $a^{2i} x_i \alpha^{2i}$ ill behaved, but we need $G\ne 1$ 
to guarantee that the products $\prod [x_i,a]$ and $\prod[x_i,\alpha]$
do not disappear completely.   
\end{proof}

\begin{proof}[Proof of Addendum~\ref{adden1}] 
Construct $P$ as above except add an additional relation in (iii) with 
$i=m+1$ and $x_{m+1} =1$. 
Everything is unchanged except that if $G\ne1$ then the amalgamating 
subgroup $E$ in the amalgamated free product decomposition of $P$ is now a 
free group of rank $m+5$.
The Mayer-Vietoris sequence of this amalgamated free product decomposition  
gives an exact sequence 
$$H_2 (P) \longrightarrow \zed^{m+5} \longrightarrow H_1(G) \oplus \zed^4\ .$$
Since $H_1(G)$ is generated by $m$ elements, it follows that 
$H_2 (P)$ is infinite.
\end{proof} 

The application of Proposition~\ref{prop1} to our recognition problems will 
make use of the particular construction employed in the proof. 
Here we pause to note that statements (1) and (2) of Proposition~\ref{prop1} 
alone quickly yield the following embedding theorem.

\begin{thm}\label{thm:embed}
There is a computable function which takes a finite presentation of a group 
$G$ and produces a finite presentation of a group $K\in \K_3$ 
and an embedding of $G$ in $K$.
\end{thm}

\begin{proof} 
By Proposition~\ref{prop1} (1) and (2) we can construct a finite 
presentation of a perfect group $P$ in which $G$ embeds. 
Let $K$ be the iterated HNN extension of $P\times P$
$$\langle P \times P, \ s,t,u :
s^{-1} (1,p) s = (p,1)\ ,\ p\in P\ ,\ 
t^{-1} (1,p) t = (p,p)\ ,\ p\in P\ ,\ 
u^{-1} su = s^2\ ,\ u^{-1} tu = t^2\rangle\ .$$
Note that after the first two HNN extensions, the stable letters $s,t$ 
are a basis for a free subgroup of rank~2.

Since $H_1 (P) =0$, we clearly have $H_1 (K) \cong\zed$. 
Also, the Mayer-Vietoris sequence for HNN extensions implies that $H_2 (K)=0$.
Finally, $K = \ll\! u\!\gg$. 
Hence $K\in \K_3$ by Theorem~\ref{thm:Kervaire}. 
Since $G$ embeds in $P$, it embeds in $K$.
\end{proof}

\begin{cor}\label{cor:iso-copy}
There is a group $K\in \K_3$ which contains an isomorphic copy of every 
finitely presented group.
\end{cor}

\begin{proof}
This follows from Theorem~\ref{thm7} and Higman's theorem that there 
exists a finitely presented group which contains an isomorphic copy of 
every finitely presented group \cite{Hig}.
\end{proof}

\begin{cor}\label{cor:wordprob}
There is a group $K\in \K_3$ with unsolvable word problem.
\end{cor}

Corollary~\ref{cor:wordprob} will be used in Section~7 to show that the 
triviality problem for $n$-knots, $n\ge  3$, is unsolvable. 

We prove the unsolvability of the various recognition problems that we 
consider in this section by showing that their solvability would imply 
the solvability of $\Rec (\G,\{1\})$. 
The proofs all follow the same pattern, which can be described in the 
following way. 
Suppose $\C\subset \B \subset \A$ ($\subset \G$). 
An {\em $(\A,\B,\C)$-construction\/} is a computable function 
$$f :(P(\G), P(\{1\}), P(\G- \{1\}) 
\to (P(\A), P(\C), P(\A-\B))\ .$$
In words, an $(\A,\B,\C)$-construction is an effective procedure which 
takes an arbitrary finite presentation of a group $G$ and produces a finite 
presentation of a group $A\in \A$, such that if $G=1$ then $A\in \C$ and 
if $G\ne 1$ then $A\notin \B$. 
Hence if $\D$ is any class of groups such that $\C \subset \D\subset \B$, 
then $A\in  \D$ if and only if $G=1$. 
It follows from the unsolvability of $\Rec (\G,\{1\})$ that if an 
$(\A,\B,\C)$-construction exists (and $P(\A)$) is recursively enumerable) 
then $\Rec (\A,\D)$ is unsolvable. 
We remark that all our $(\A,\B,\C)$-constructions will actually produce an 
embedding of $G$ in $A$. 

\begin{thm}\label{thm7} 
$(\K_3,\K_2,\{\zed\})$-constructions exist.
\end{thm}

\begin{proof}
Given a finite presentation of a group $G$, we must produce a finite 
presentation of a group $K\in \K_3$, such that if $G=1$ then $K\cong\zed$ 
and if $G\ne 1$ then $K\notin \K_2$. 

Let $P$ be the finitely presented group described in the proofs of 
Proposition~\ref{prop1} and Addendum~\ref{adden1}. 
Let $Q$ be the HNN extension 
$\langle P\times P ,\ s:s^{-1} (1,p)s = (p,1),\ p\in P\rangle$.
Let $q = (a,\alpha) \in P\times P$, and let $R$ be obtained from the free 
product $Q_1 * Q_2$  of two copies of $Q$ by adjoining the relations 
$s_1 = q_2$, $q_1 = s_2$. 
Here, a letter with subindex 1 (resp.\ 2) represents an element of the 
first (resp.\ second) copy of $Q$.
Finally, let $K = \langle R,t : t^{-1} (1,p_i) t = (p_i,p_i)$, 
$p_i\in P_i$, $i=1,2\rangle$. 
Note that we can write down a finite presentation of $K$. 

If $G=1$, then $P=1$, and hence $R=1$ and $K\cong \zed$.

{From} now on, assume that $G\ne 1$. 
We shall show that $G\in \K_3-\K_2$. 
First note that $q^n\notin (1\times P)\cup (P\times 1)$ for $n\ne 0$,
and so, by Britton's Lemma, the subgroup $\langle s,q\rangle$ of $Q$ is 
a free group of rank~2. 
Hence $R$ is a free product with amalgamation $Q_1 *_{F_2} Q_2$. 
Since, in $Q$, $(1\times P)\cap \langle s,q\rangle = \{1\}$, the 
subgroup $S = \langle 1\times P_1,\ 1\times P_2\rangle$ of $R$ is the 
free product $(1\times P_1) * (1\times P_2)$. 
Also, the map $\delta :S\to R$ given by $\delta (1,p_i) = (p_i,p_i)$, 
$p \in P_i$, $i=1,2$, is a monomorphism, since, if $\Delta$ is the diagonal 
subgroup of $P\times P$, then $\Delta \cap \langle s,q\rangle =\{1\}$ in $Q$.
Hence $K$ is an HNN extension of $R$. 
Note that  $\LL t\GG =K$, and that $H_1 (K) \cong\zed$. 

Let $\iota :S\to R$ be the inclusion map. 
Using the Mayer-Vietoris sequences for free products with amalgamation and 
HNN extensions \cite{Bie}
one sees that $\iota_* :H_2(S) \to H_2(R)$ is an isomorphism.
Also, for $x\in H_2(S)$, $\delta_* (x) = 2\iota_* (x)$. 
Hence, again using the Mayer-Vietoris sequence for HNN extensions, we obtain 
$H_2(K)=0$. 
Hence $K\in \K_3$. 

To see that $K\notin \K_2$, we examine $H_j (K')$, $j=1,2$, where $K'$ is 
the commutator subgroup of $K$. 
Consider spaces $X_R$, $X_S$, where $X_H$ denotes an aspherical complex 
with basepoint $*$ and $\pi_1 (X_H,*)\cong H$. 
Let $f,g : (X_S,*) \to (X_R,*)$ be cellular maps inducing, respectively, 
the maps $\iota$ and $\delta$ on fundamental groups. 
In the disjoint union of $X_R$ and $X_S\times [0,1]$, identity 
$(x,0) \in X_S \times [0,1]$ with $f(x)$ and $(x,1)\in X_S\times [0,1]$ 
with $g(x)$, obtaining an aspherical complex $X_K$.
Then $H_* (K') \cong H_* (\wtilde X_K)$, where $\wtilde X_K$ is the 
universal abelian (infinite cyclic) covering of $X_K$. 
As in \cite[p.43]{L} one gets an exact sequence 
$$\cdots \too H_j (S) \otimes_\zed \Lambda \xrightarrow{\ d\ } 
H_j(R)\otimes_\zed \Lambda \too H_j(K') \too H_{j-1}(S)\otimes_{\zed} 
\Lambda \too \cdots\ ,
$$
where $\Lambda =\zed [t,t^{-1}]$ is the integral group ring of the 
infinite cyclic group generated by $t$ and $d$ is given by 
$d(x\otimes\lambda) = \delta_* (x) \otimes t\lambda - \iota_* (x)\otimes
\lambda$.

Note that $d: H_0(S)\otimes_{\zed} \Lambda \to H_0(R)\otimes_{\zed} \Lambda$ 
can be identified with $t-1:\Lambda \to\Lambda$, which is injective. 
Since $R$ is perfect, it follows that $H_1 (K') =0$. 

Recall that $\iota_* : H_2 (S) \to H_2 (R)$ is an isomorphism, and that, for 
$x\in H_2(S)$, $\delta_* (x) = 2\iota_* (x)$. 
Hence the exact sequence above shows that $H_2(K')$ is 
isomorphic to the cokernel of the map 
$d' :H_2 (S)\otimes_\zed \Lambda \to H_2(S)\otimes_\zed \Lambda$ defined by 
$d' (x\otimes \lambda) = x\otimes (2t-1)\lambda$. 
Thus $H_2(K') \cong H_2 (S)\otimes (\Lambda/(2t-1) \Lambda) \cong 
H_2 (S)\otimes \zed [1/2]$. 
Since $H_2(S) \cong H_2 (P)\oplus H_2 (P)$ is infinite (Addendum~\ref{adden1})
it follows that $H_2 (K';\que) \ne0$. 
This, together with the fact that $H_1 (K';\que) =0$, implies that 
$K\notin \K_2$ 
\cite{Fa}, \cite{Hi}, \cite{Miln1}.
\end{proof}

\begin{cor}\label{cor3.4}
If $\K_0 \subset \B\subset\K_2$ then $\Rec (\K_3,\B)$ is unsolvable.
\end{cor}

\begin{thm}\label{thm8}
$(\S,\K_3,\{\zed\})$-constructions exist. 
\end{thm}

\begin{proof}
Let $G = \langle x_1,\ldots,x_m :r_1,\ldots,r_n\rangle$ be a finitely 
presented group. 
Embed $G$ in a perfect group $P$ as in the proof of Proposition~\ref{prop1}. 
Consider the groups $A_5 = \langle c,d : c^2 = d^3 = (cd)^5 =1\rangle$ and 
$\zed_2 = \langle e: e^2 =1\rangle$. 
Let $Q$ be the group obtained from $P * A_5 * \zed_2$ by adding the relation
$b=de$. 
Let $\wtilde Q$ be the universal central extension of the perfect group
$Q$ (see \cite{Miln2}). 
Then $\wtilde Q$ has a presentation with generators $x_1,\ldots,x_m$, 
$a,\alpha,b,\beta,c,d,e$, and relations (i) through (v) of the proof of 
Proposition~\ref{prop1}, together with $c^2 = d^3 = (cd)^5$, $b=de$, and 
$[r,g] =1$ where $r$ runs over the words $r_1,\ldots,r_n$, $c^2$, $e^2$ and 
$g$ runs over the generators of $\wtilde Q$. 
(Compare the proof of Lemma~2 in Section~10 of \cite{KVF}.)
The kernel of the natural epimorphism from $\wtilde Q$ to $Q$ is 
(contained in) the center of $\wtilde Q$; also $H_2(\wtilde Q)=0$. 
Now adjoin to $\wtilde Q$ the relation $c^2=1$ to get the group $R$.
Let $K = \zed\times R$. 
It is not difficult to verify that if $G=1$ then $K\cong \zed$. 

Assume in the rest of the proof that $G\ne1$; we claim that $K\in \S-\K_3$.
First note that $c^2$ is a central element of order~2 in $\wtilde A_5 = 
\langle c,d :c^2 = d^3 = (cd)^5\rangle$ and that $c^2 = [c,(dcd^{-1}c)^2d]$
in $\wtilde A_5$ and therefore in $\wtilde Q$. 
To see that  $c^2$ is non-trivial in $\wtilde Q$, adjoin to $\wtilde Q$ 
the relations $r_1^2 =1,\ldots, r_n^2=1,\ e^2=1$ to obtain the iterated 
free product with amalgamation $S = (P\times \zed_2) *_{\zed\times\zed_2} 
(\wtilde A_5 *_{\zed_2} (\zed_2\times\zed_2))$. 
Here $\wtilde A_5$ and $\zed_2\times\zed_2 = \langle y,e : y^2=e^2 = [y,e]
=1\rangle$ are amalgamated by $c^2=y$; $P\times\zed_2$ and 
$\wtilde A_5 *_{\zed_2} (\zed_2\times\zed_2)$ are amalgamated by 
$b=de$, $z=c^2$ where $z$ generates the second factor of $P\times\zed_2$.
Since $c^2$ is non-trivial in $\wtilde A_5$, it is non-trivial in $S$ 
and therefore in $\wtilde Q$.

Hence we have a short exact sequence $1\to \zed_2 \to \wtilde Q\to R\to 1$ 
with $\zed_2$ contained in the center of $\wtilde Q$.
The associated 5-term exact sequence then yields $H_2 (R) \cong\zed_2$; 
therefore $H_2 (K)\cong\zed_2$ and so $K\notin \K_3$.

To see that $K\in \S$ first note that, if $\tau$ is a generator of $\zed$, 
then $(\zed\times\wtilde Q)/\LL \tau c\GG =1$. 
Since, in addition, $H_2 (\zed\times \wtilde Q)=0$, $\zed\times\wtilde Q$ 
has a Wirtinger presentation with a generator representing $\tau c$ 
(see \cite{Y} or \cite{Si}).
Finally, $K$ is obtained from $\zed\times \wtilde Q$ by adding the 
relation $[\tau c,(dcd^{-1}c)^2 d]$, so $K$ also has a Wirtinger presentation.
\end{proof}

\begin{cor}\label{cor4} 
If $\K_0 \subset \B \subset \K_3$ then $\Rec (\S,\B)$ is unsolvable.
\end{cor}

\begin{thm}\label{thm9}
$(\M,\S,\{\zed\})$-constructions exist. 
\end{thm}

\begin{proof}
First embed $G$ in a perfect group $P$ as in (the proof of) 
Proposition~\ref{prop1}. 
Let $K = \langle P,s: s^{-1} bs=b^2\rangle$. 
Then $\LL s\GG =K$, and so $K\in \M$.

If $G=1$ then $P=1$ and $K\cong \zed$.

Now assume $G\ne 1$. 
Then $K$ is an HNN extension of $P$, and the Mayer-Vietoris sequence 
of this extension gives an exact sequence
$$H_2(K) \too H_1(\zed) \too H_1(P)=0\ .$$
Hence $H_2(K) \ne0$. 
This already shows that $K\notin \K_3$.
To show that $K\notin \S$ we use Theorem~\ref{thm2}.

Let $t\in K$ be an element such that $\LL t\GG =K$, and let $c\in C_t$, 
the centralizer of $t$ in $K$. 
Then $\langle t,c\rangle$ either is isomorphic to $\zed$ or does not split 
non-trivially as a free product with amalgamation or HNN extension.
If the latter holds then (see \cite[Corollary~3.8]{SW}) 
$\langle t,c\rangle$, and therefore $t$, lies in a conjugate of $P$, 
contradicting our assumption that $\LL t\GG =K$. 
Hence $\langle t,c\rangle \cong \zed$, and so $t\wedge c=0$. 
Since $H_2(K) \ne0$, Theorem~\ref{thm2} implies that $K\notin \S$.
\end{proof}

\begin{cor}\label{cor5}
If $\K_0 \subset \B \subset \S$ then $\Rec (\M,\B)$ is unsolvable.
\end{cor}

\section{Having weight 1 is unrecognizable}

Let $U = \langle u_1,u_2 : \ldots\rangle $ be a 2-generator, finitely 
presented, torsion-free group with unsolvable word problem. 
Such a group exists by \cite{Bo} or \cite{Mill}.

Let $K$ be the iterated HNN extension
$$\langle U,y,y_2,z : y_i^{-1} u_i y_i = u_i^2\ \ (i =1,2),\ \ 
z^{-1} y_i z= y_i^2\ \ (i=1,2)\rangle\ .$$
$K$ is still torsion-free and is normally generated by $z$. 
Also, for any non-trivial element $w$ of $U$, the subgroup 
$\langle z,w\rangle$ of $K$ is isomorphic to $F_2$, the free group of rank~2.

Consider the group $Q = \langle r,s,t: s^{-1} rs =r^2,\ t^{-1}st=s^2\rangle$.
It is torsion-free, normally generated by $t$, and the subgroup
$\langle r,t\rangle \cong F_2$.

For any word $w$ in $u_1,u_2$, let $D_w$ be obtained from the free product 
$K*Q$ by adding the relations $w=t$, $z=r$. 
If $w$ represents the trivial element of $U$, then $D_w=1$, while if $w$ 
does not represent the trivial element of $U$ then $D_w$ is a free product 
with amalgamation $K *_{F_2} Q$, and hence is torsion-free and non-trivial.
Let $G_w = \zed * D_w$. 
Then, by Klyachko's theorem \cite{Kl}, $G_w$ has 
weight~1 if and only if $w$ represents the trivial element of $U$. 
Thus we have proved 

\begin{prop}\label{prop4.1} 
{\rm (1)} If $w$ represents the trivial element of $U$ then $G_w\cong\zed$;

{\rm (2)} if $w$ does not represent the trivial element of $U$ then 
$G_w\notin \M$.
\end{prop}

\noindent Since $U$ has unsolvable word problem we get 

\begin{cor}\label{cor4.2}
If $\K_0 \subset \B\subset \M$ then $\Rec (\G,\B)$ is unsolvable.
\end{cor}

\section{Unsolvable problems about homology, Whitehead groups\\
and surgery groups}

By the Poincar\'e Conjecture 
\cite{Pe1}, \cite{Pe2}, \cite{Pe3}  
and the recognizability of the
3-sphere \cite{Ru}, it follows that there is an algorithm which decides 
whether or not a given closed 3-manifold is 1-connected.
It is interesting to note that one can phrase this in terms of homology 
of groups. 
Let $\A rt$ be the set of ordered presentations $\langle x_1,\ldots,x_n :
r_1,\ldots,r_n\rangle$ such that $\prod_{i=1}^n r_i x_i r_i^{-1} = 
\prod_{i=1}^n x_i$ in the free group with generators $x_1,\ldots, x_n$. 
The groups defined by the members of $\A rt$ are precisely the fundamental 
groups of closed orientable 3-manifolds. 
(This follows from the fact that every closed orientable 3-manifold is an
open book with planar pages (see e.g. 
\cite[pages~340--341]{Ro}) and a theorem of 
Artin (see \cite[Theorem~1.9]{Bir}; see also \cite{Wi} and \cite{Gon3}). 
Thus the question of deciding whether a closed 3-manifold is 1-connected is 
equivalent to the question  of deciding whether a member of $\A rt$ presents 
the trivial group.
This in turn can be phrased in terms of homology, as follows. 
Let $M$ be a closed orientable 3-manifold and let 
$M_1 \#\cdots \# M_r \# N_1 \#\cdots \# N_s\# S^1 \times S^2 \# \cdots \# 
S^1 \times S^2$ 
be the connected sum decomposition of $M$ into prime manifolds \cite{Miln3}, 
where $\pi_1 (M_i)$ is infinite non-cyclic, $1\le i\le r$, and $\pi_1(N_j)$ 
is finite, $1\le j\le s$. 
Let $n_j$ be the order of $\pi_1 (N_j)$, $1\le j\le s$. 
Then $H_3 (\pi_1 (M)) \cong \zed^r \oplus \zed_{n_1} \oplus \cdots\oplus 
\zed_{n_s}$, and so $\pi_1 (M) = 1$ if and only if $H_1 (\pi_1 (M)) =0$ 
and $H_3 (\pi_1(M))=0$. 
Thus the simple-connectedness problem is equivalent to deciding, for members 
$\A$ of $\A rt$, whether the finitely generated abelian groups $H_1(\A)$ 
and $H_3 (\A)$ are trivial.
(Here, and in the sequel, if $F$ is a functor defined on the category of 
groups, and $\P$ is a group presentation, we abbreviate 
$F(|\P|)$ to $F(\P)$.) 
As noted above, it is known (albeit indirectly) that this decision problem 
is solvable.

However, it is natural to ask the question for the class of all finite 
presentations.
We shall see that this and many other problems concerning the computation 
of the homology of groups in dimensions greater than 1 are algorithmically 
unsolvable.
We will also prove incomputability results about Whitehead groups 
$Wh_n(G)$ and Wall's surgery groups $L_n(G)$.
$H_*(G)$ will denote the infinite sequence 
$(H_1(G),H_2(G),H_3(G),\ldots)$ of integral homology groups of the group $G$.

\begin{thm}\label{thm11}
Let $\C$ be a class of infinite sequences $(A_1,A_2,A_3,\ldots)$ of abelian 
groups  which is closed under isomorphisms.\footnote{$(A_1,A_2,A_3,\ldots)$ 
is isomorphic to $(A'_1,A'_2,A'_3,\ldots)$ if $A_i \cong A'_i$ for all $i$.}
Suppose there are finitely presented groups $G_1,G_2$ such that 
$H_1 (G_1)\cong H_1(G_2)$, $H_* (G_1)\in \C$ and $H_*(G_2)\notin \C$. 
Then the set of finite presentations $\P$ such that $H_*(\P)\in \C$ is not 
recursive. 
\end{thm}

\begin{quest}
When can one replace recursive by recursively enumerable?
\end{quest}

\begin{rem}\label{rem1}
For a class $\C$, closed under isomorphisms, which does not satisfy the 
hypothesis of the theorem, a finite presentation $\P$ has integral 
homology belonging to $\C$ if and only if $H_1(\P)\in \C_1$ where $\C_1$ is the 
class of finitely generated abelian groups which are first terms of sequences 
belonging to $\C$; hence $\{\P :H_* (\P) \in \C\}$ is recursive if and
only if $\C_1$ is recursive.
\end{rem}

\begin{proof}[Proof of Theorem~\ref{thm11}]
Call a finite presentation $\langle x_1,\ldots,x_m : r_1,\ldots, r_n\rangle$
{\em freely related\/} if $r_1,\ldots,r_n$ generate a free group of rank~$n$
in the free group on $x_1,\ldots,x_m$. 
Clearly every finitely presented group has a freely related presentation. 
If we have a freely related presentation of deficiency $d$ of a group  we 
can find a freely related presentation of deficiency $d-1$ of the same 
group by adjoining, for example, new generators $z_1,z_2$ and relators 
$z_1,z_2^3$ and $z_2 z_1 z_2$.

It follows that $G_1$ and $G_2$ have freely related presentations with the 
same deficiency $d$. 
Writing $s= \dim H_1 (G_1;\que) = \dim H_1(G_2;\que)$, let $G= G_2$ 
(resp.\ $G_1$) if the sequence $(H_1(G_1),\zed^{s-d},0$, $0,\ldots)$ 
belongs to $\C$ (resp.\ does not belong to $\C$). 

Let $\langle x_1,\ldots,x_m : r_1,\ldots, r_n\rangle$ be a freely related 
presentation of $G$ of deficiency $d$.
Let $\U = \langle \mu_1,\ldots,\mu_p :\ldots\rangle$ 
be a finite presentation of an 
acyclic (i.e. with trivial integral homology in all positive dimensions)
group $U$ with unsolvable word problem. 
Such a group exists by \cite{N} (or \cite{Bo}), \cite[Theorem~E]{BDM} 
and \cite{R2}. 
Consider also a finite presentation $\Y = \langle y_1,\ldots,y_n,\ldots,
y_q :\ldots\rangle$ of an acyclic group $Y$ such that $y_1,\ldots,y_n$ 
represent $n$ different non-trivial elements of $Y$. 
Denote by $\P_m * \U * \Y$ the presentation whose generators are 
$x_1,\ldots,x_m$, $\mu_1,\ldots,\mu_p$, $y_1,\ldots,y_q$ and whose relators 
are those of $\U$ and $\Y$.

To a word $w$ in the generators $\mu_1,\ldots,\mu_p$ of $\U$ we associate 
the presentation $\Pi_w$ obtained by adjoining to $\P_m * \U * \Y$ the 
relations $r_i = [w,y_i]$, $i=1,\ldots,n$.
If $w=1$ in $U$ then $\Pi_w$ presents $G * U * Y$ so $H_* (\Pi_w)\cong H_*(G)$.
If $w\ne 1$ in $U$ then $[w,y_1],\ldots,[w,y_n]$ (resp.\ $r_1,\ldots,r_n$) 
generate a free group of rank~$n$ in $U*Y$ (resp.\ in the free group $F_m$ 
on $x_1,\ldots,x_m$) so that $\Pi_w$ presents a free product of $F_m$ and 
$U*Y$ amalgamated along a free group of rank~$n$; the Mayer-Vietoris 
sequence for free products with amalgamation then yields $H_i (\Pi_w) =0$ 
for $i>2$, $H_2 (\Pi_w) \cong\zed^{s-d}$ and $H_1(\Pi_w)\cong H_1(G)$. 
Since precisely one of the sequences $H_*(G)$, $(H_1(G),\zed^{s-d},0,0,\ldots)$
belongs to $\C$, it follows that an algorithm which decides whether or not 
groups given by finite presentations have an integral homology sequence 
which belongs to $\C$ could be used to solve the word problem for $U$. 
Since $U$ has unsolvable word problem, the existence of such an algorithm 
is impossible. 
Thus, the set of finite presentations $\P$ with $H_*(\P) \in \C$ is not 
recursive.
\end{proof} 

\begin{rem}\label{rem2} 
Recall that a property $P$ of (isomorphism classes of) finitely 
presented groups is a {\em Markov property\/} if there exist finitely 
presented groups $G_1$ and $G_2$ such that 
\begin{itemize}
\item[(1)] $G_1$ has property $P$; and 
\item[(2)] if $G_2$ embeds in a finitely presented group $H$ then $H$ 
does not have property $P$.
\end{itemize} 
If $\C$ is an isomorphism closed class of sequences of abelian groups 
and if $H_* (G)\in \C$ for some  finitely presented group $G$ then ``having
a homology sequence which belongs to $\C$'' is not a Markov property 
since any finitely presented group embeds in a finitely presented acyclic 
group $A$ (\cite{BDM}) and therefore in $A*G$, a group whose homology 
belongs to $\C$.
Therefore Theorem~\ref{thm11} cannot be derived from 
Rabin's theorem (Theorem~1.1 of \cite{R1}).
\end{rem}

\begin{cor}\label{cor6} 
If $I$ is a set of natural numbers containing a number greater than $1$ 
then the set of finite presentations $\P$ such that $H_i(\P)=0$ for 
every $i\in I$ is not recursive.
\end{cor}

\begin{proof} 
Take $G_1$ to be the trivial group and, if $n\in I-\{1\}$, take 
$G_2 = A_5^n \times SL (2,5)$.
\end{proof}

The case $I= \{1,3\}$ is the one which we were discussing above in relation 
with the simple-connectedness problem for 3-manifolds.

The case $I = \{1,2\}$ corresponds to the problem of deciding whether or 
not a finitely presented group is the fundamental group of a smooth 
homology $n$-sphere, $n\ge 5$, that is (see \cite{Ke2}), a group with 
trivial first and second homology.
This problem is, therefore, unsolvable.

We now prove an incomputability result for $Wh_0$ and $Wh_1$, 
where $Wh_0 (G) = \wtilde K_0(\zed G)$, the reduced projective class 
group (\cite[page~419]{Miln4}), and $Wh_1(G)$ is the usual Whitehead group 
(\cite[page~372]{Miln4}). 

\begin{thm}\label{thm12} 
Let $\C$ be a class of pairs $(A_0,A_1)$ of 
abelian groups which is closed under isomorphisms. 
Suppose $(0,0)\in \C$ and $(Wh_0 (G),Wh_1(G)) \notin \C$ for some finitely 
presented group $G$.  
Then the set of finite presentations $\P$ such that $(Wh_0(\P),Wh_1(\P))\in \C$ 
is not recursive.
\end{thm}

\begin{proof} 
By \cite{W3}, for a free product with amalgamation $H=A *_F B$ with $F$ 
free one has a Mayer-Vietoris sequence 
$$Wh_1 (F) \to Wh_1 (A) \oplus Wh_1(B) \to Wh_1 (H) 
\to Wh_0(F) \to Wh_0 (A) \oplus Wh_0 (B) \to Wh_0(H)\to 0\ ,$$
and $Wh_0(F) = Wh_1 (F)=0$.

In \cite[Chapter 12]{Rot} a sequence of finitely presented groups 
$G_1,G_2,\ldots,G_s$ is constructed such that $G_1$ is free, $G_s$ 
has unsolvable word problem and, for $1\le i <s$, $G_{i+1}$ is 
an HNN extension of $G_i$ along a free group; see \cite{CM}. 
Therefore $G_s$ belongs to Waldhausen's class $Cl$ in \cite{W3} and hence 
$(Wh_0 (G_s), Wh_1(G_s)) = (0,0)$. 
Let $U = G * G_s$. 
Then $U$ has unsolvable word problem and, using the Mayer-Vietoris sequence 
above, $(Wh_0(U),Wh_1(U)) = (Wh_0(G),Wh_1(G)) \notin \C$.

Let $\Pi = \langle x_1,\ldots,x_m :r_1,\ldots,r_n\rangle$ be a finite 
presentation of $U$, and let $w$ be a word in $x_1,\ldots,x_m$. 
Let $\Pi_w$ be the presentation obtained from $\Pi$ by adjoining 
additional generators $a,\alpha,b,\beta$ and additional relations
\begin{equation*}
\begin{split} 
a\alpha a^{-1} & = b^2\\
\alpha a\alpha^{-1} & = b\beta b^{-1}\\
a^{2i} x_i \alpha^{2i} & = \beta^{2i+2} b\beta^{-2i-2}\qquad 1\le i\le m\\
[w,a] & = \beta^2 b\beta^{-2}\\
[w,\alpha] & = \beta b\beta b^{-1} \beta^{-1}
\end{split}
\end{equation*}
as in \cite{Gor2}. 

If $w=1$ in $U$, then $\Pi_w$ presents a trivial group so that 
$(Wh_0 (\Pi_w),Wh_1(\Pi_w))  = (0,0) \in \C$. 

If $w\ne 1$ in $U$, then $\Pi_w$ presents a free product with 
amalgamation $(U*F_2) *_{F_{m+4}}  F_2$ 
(where $F_r$ is a free group of rank~$r$).
The Mayer-Vietoris sequence above and the fact that free groups have 
trivial $Wh_0$ and $Wh_1$ implies that if $w\ne 1$ in $U$ then 
$(Wh_0 (\Pi_w),Wh_1(\Pi_w)) \cong (Wh_0(U),Wh_1(U)) \notin \C$. 

Since the set of words $w$ which represent the trivial element of $G$ is 
not recursive, it follows that the set of finite presentations 
$\P$ such that $(Wh_0(\P),Wh_1(\P))\in \C$ is not recursive.
\end{proof}

\begin{cor}\label{cor7} 
Let $i= 0$ or $1$. 
Then the set of finite presentations $\P$ such that $Wh_i(\P) =0$ is 
not recursive.
\end{cor}

\begin{proof} 
There is a finitely presented group $A$ whose $i$-th Whitehead group is 
non-trivial (for example, $\zed_{23}$ for $i=0$ (\cite[page~419]{Miln4}), and 
$\zed_5$ for $i=1$ (\cite[page~374]{Miln4}). ~\qed
\medskip

Finally, we turn to surgery groups (\cite{Wa}). 
Let $L_n^h(G)$ (resp.\ $L_n^s(G)$) denote Wall's group of surgery obstructions
for the problem of obtaining homotopy equivalences (resp.\ simple homotopy 
equivalences) for orientable manifolds of dimension~$n$ and fundamental 
group $G$. 
For $x=h$ or $s$, $L_n^x$ is a functor from groups to abelian groups with 
$L_n^x = L_{n+4}^x$. 
Write $L_n^x (G) = \wtilde L_n^x (G) \oplus L_n^x (1)$.
Note that $L_n^h (G) \otimes \que \cong L_n^s (G)\otimes\que$ by the 
Rothenberg exact sequence. 
(See Section~17D of \cite{Wa}.) 
\renewcommand{\qed}{}
\end{proof}

\begin{thm}\label{thm13} 
Let $n\ge 0$ and $x=h$ or $s$. 
Then the set of finite presentations $\P$ such that $\wtilde L_n^x(\P)=0$ 
is not recursive.
\end{thm}

\begin{proof} 
Let $U$ be a 2-generator, finitely presented group with unsolvable word 
problem (see \cite[Chap.~IV, Thm.~3.1]{LS}). 
Let $G = U*\zed^6$ and let $\Pi = \langle x_1,\ldots,x_8 : r_1,\ldots, r_q
\rangle$ be a finite presentation of $G$. 
If $w$ is a word in $x_1,\ldots,x_8$, let $\Pi_w$ be the  presentation
defined in the proof of Theorem~\ref{thm12}. 

If $w=1$ in $G$ then $\wtilde L_n^x (\Pi_w) = \wtilde L_n^x(1) =0$. 

If $w\ne 1$ in $G$ then $\Pi_w$ presents a free product with amalgamation 
$(G*F_2) *_{F_{12}} F_2$ so, from \cite[Corollary~6]{C2}, we obtain an 
exact sequence 
$$L_n^x (F_{12}) \otimes \que \too (L_n^x (G *F_2) \oplus L_n^x (F_2)) 
\otimes \que \too L_n^x (\Pi_w) \otimes \que\ .$$

By \cite[Theorem~16]{C1}, $\dim L_n^x (F_{12}) \otimes\que \le 12$ and, 
using Corollary~6 of \cite{C2}, Corollary~15 and Theorem~16 of \cite{C1} 
one sees that $\dim (L_n^x (U*\zed^6 * F_2) 
\oplus L_n^x (F_2))\otimes\que \ge16$
so that $\dim L_n^x (\Pi_w) \otimes \que \ge 4$ and 
$\dim \wtilde L_n^x (\Pi_w) \otimes \que \ge 3$. 
Thus for $w\ne 1$ in $G$ $\wtilde L_n^x (\Pi_w)$ is non-trivial. 

As above, the non-recursiveness of the set of words representing the 
trivial element of $G$ implies the non-recursiveness of the set of finite 
presentations $\P$ with $L_n^x (\P)=0$.
\end{proof}

\section{Enumeration of knots}

In this section we define presentations of (locally flat PL) $n$-knots and show 
that they can be recursively enumerated. 
A presentation will be a description of a knot type in finite terms.

Any abstract (simplicial) complex considered, $\bA$, will be assumed to have 
as its set $V(\bA)$ of vertices a finite set of natural numbers. 
Any simplicial complex $A$ will be finite, its set of vertices will be 
denoted by $V(A)$ and its underlying polyhedron by $|A|$.

When we consider pairs $(A,B)$ (resp.\ $(\bA,\bB)$) of simplicial (resp.\ 
abstract) complexes, $B$ (resp.\ $\bB$) is a subcomplex of $A$ (resp.\ $\bA$) 
and $\overline{A-B}$ (resp\ $\overline{\bA -\bB}$) denotes the smallest 
subcomplex of $A$ (resp.\ $\bA$) containing $A-B$ (resp.\ $\bA-\bB$). 

A {\em realization\/} $(A,\phi)$ of the abstract complex $\bA$ is a simplicial 
complex $A$ together with a bijection $\phi: V(\bA)\to V(A)$ such that, 
a subset $s$ of $V(\bA)$ is a simplex of $\bA$ if and only if
 the convex hull of $\phi(s)$ is a simplex of $A$. 
We also say that $A$ is a realization of $\bA$. 

If $(A,\phi)$ is a realization of $\bA$ and $\bB$ (resp.\ $B$) is a subcomplex 
of $\bA$ (resp.\ $A$) such that $(B,\phi\, |\,V(\bB))$ is a realization of 
$\bB$ then we say that $(A,B)$ is a realization of $(\bA,\bB)$. 
Notice that if $(A_1,B_1)$, $(A_2,B_2)$ are two realizations of $(\bA,\bB)$ 
then $(|A_1|,|B_1|) \approx (|A_2|,|B_2|)$. 
($\approx$ denotes PL homeomorphism.)

Let $\bA_1$, $\bA_2$ be two abstract complexes with realizations $A_1$, $A_2$
respectively.
$\bA_1$ is {\em equivalent to\/} $\bA_2$ (we write $\bA_1\sim \bA_2$) if 
$|A_1| \approx |A_2|$.

\begin{defn}\label{defn1}
A {\em presentation of an $n$-knot\/} is a pair $(\bA,\bB)$ of abstract 
complexes having a realization $(A,B)$ such that $|A| \approx S^{n+2}$ 
and $|B| \approx S^n\times D^2$.
\end{defn}

We will see that a presentation defines a unique knot type. 

If $T$ is a polyhedron PL homeomorphic  to $S^p\times D^q$ a {\em core\/} 
of $T$ is the image of $S^p\times \{0\}$ under a PL homeomorphism from 
$S^p\times D^q$ onto $T$.

\begin{lem}\label{lem3} 
Let $T$ be PL homeomorphic to $S^n\times D^2$ and let $K$, $K'$ be two 
cores of $T$. 
Then there is a PL homeomorphism from $T$ onto $T$, mapping $K$ onto $K'$, 
which is the identity on $\partial T$.
\end{lem} 

\begin{proof} 
We may assume $T = S^n\times D^2$ and $K = S^n\times \{0\}$.
Let $f: S^n\times D^2\to T$ be a PL homeomorphism mapping $S^n\times\{0\}$ 
onto $K'$.

Now, if $n\ge2$, the proof of Theorem 2 of \cite{Sw} shows that, if a PL 
autohomeomorphism $h$ of $S^n\times \partial D^2$ can be extended to a 
PL autohomeomorphism of $S^n\times D^2$, then it can be extended to a 
PL autohomeomorphism of $(S^n\times D^2,S^n\times \{0\})$ 
(both conditions being equivalent to the vanishing of the second 
Stiefel-Whitney class of $S^n\times D^2 \cup_h S^n\times D^2$). 
For $n=1$ this fact is well known.

Hence $f\mid \partial (S^n\times D^2)$ can be extended to a PL 
homeomorphism $g$ mapping $K$ onto itself.
Then $fg^{-1}$ maps $K$ to $K'$ and is the identity on $\partial T$.
\end{proof}

Let $(\bA,\bB)$ be a presentation of an $n$-knot. 
If $(A,B)$ is a realization of $(\bA,\bB)$ then the knot type represented 
by $(|A|,K)$, where $K$ is a core of $|B|$, is {\em the knot type 
presented by\/} $(\bA,\bB)$. 
This is well-defined because, if $K'$ is another core of $|B|$, there is, 
by the previous lemma, an autohomeomorphism of $|A|$, which is the identity 
on $|A|-|B|$, mapping $K$ onto $K'$. 
{\em The group of the $n$-knot presentation\/} $(\bA,\bB)$ is the group of 
a knot in the type presented by $(\bA,\bB)$, that is, $\pi_1 (|A|-|B|)$, 
where $(A,B)$ is a realization of $(\bA,\bB)$.

Next, we want to give a recursive enumeration of presentations.

A simplicial complex $B$ is a subdivision of the simplicial complex $A$ if 
every vertex of $A$ is a vertex of $B$ and every simplex of $B$ is 
contained in a simplex of $A$.

If $B$ is a subdivision of $A$, $(B,\phi)$ is a realization of the abstract 
complex $\bB$ and $\bA$ is the abstract complex consisting of the family of 
subsets $s$ of $V(\bB)$ such that the convex hull of $\phi (s)$ is a 
simplex of $A$, then we say that $\bB$ is a subdivision of $\bA$.

The following proposition is the Corollary to Lemma~1 of \cite{BHP}.

\begin{prop}\label{prop2}
There is a recursive function $X(\bA,k)$, $\bA$ ranging over all finite 
abstract complexes, $k=1,2,\ldots,$ that recursively enumerates for an 
arbitrary complex $\bA$ the subdivisions of $\bA$, i.e. for fixed $\bA$ the 
sequence $X(\bA,1) = \bA$, $X(\bA,2),\ldots$ is a recursive enumeration of 
all subdivisions of $\bA$.
\end{prop}

\begin{cor}\label{cor8} 
Let $\bA$ be an abstract complex. 
Then there  is a recursive enumeration of all abstract complexes 
equivalent to $\bA$.
\end{cor}

\begin{proof} 
Let $\bA_1,\bA_2,\ldots$ be a recursive enumeration of all abstract complexes.
Then $\bA = \bA_r$, say. 
Recursively enumerate all triples $(i,j,k)$ such that $X(\bA_r,i)$ is 
isomorphic to  $X(\bA_j,k)$. 
Let $(i_1,j_1,k_1), (i_2,j_2,k_2),\ldots$ be this enumeration. 
Eliminating repetitions in the sequence $\bA_{j_1},\bA_{j_2},\ldots$ we 
obtain a recursive enumeration of the complexes equivalent to $\bA$.
\end{proof} 

Now, for any $n$, choose one abstract complex $\bA^n$ (resp.\ $\bB^n$) with 
a realization having underlying polyhedron PL homeomorphic to $S^{n+2}$ 
(resp.\ $S^n\times D^2$). 
Let $\bA_1^n,\bA_2^n,\ldots$ (resp.\ $\bB_1^n,\bB_2^n,\ldots$) be a recursive 
enumeration of all abstract complexes equivalent to $\bA^n$ (resp.\ $\bB^n$).
{From} these two enumerations we obtain an enumeration of all pairs 
$(\bA_i^n,\bB_j^n)$ such that $\bB_j^n$ is a subcomplex of $\bA_i^n$. 
We have therefore proved: 

\begin{thm}\label{thm14} 
For any $n\ge 0$ there is a recursive enumeration of the set of all 
$n$-knot presentations.
\end{thm}

It now makes sense to talk about recursively enumerable and recursive 
sets of presentations of $n$-knots.

Here is a consequence of Theorem~\ref{thm14}.

\begin{cor}\label{cor9} 
Given a finite presentation $\Pi$ of an $n$-knot group one can find a 
presentation $\P$ of an $n$-knot whose group is isomorphic to the group 
presented by $\Pi$.
\end{cor}

\begin{proof}
Let $\P_1,\P_2,\ldots$ be an enumeration of the presentations of $n$-knots.
For every $i$ one can find a finite presentation of the group $G_i$ of the 
knot type presented by $\P_i$ and, therefore, using Tietze operations, 
recursively enumerate all finite presentations of $G_i$. 
Now, enumerate recursively all pairs $(\P_i,\Pi_j)$  such that the finite 
presentation $\Pi_j$ presents the group  of $\P_i$. 
Take the first pair $(\P_i,\Pi_j)$ in this enumeration such that 
$\Pi = \Pi_j$ and take $\P = \P_i$.
\end{proof}

As a consequence we have the following geometric version of 
Corollary~\ref{cor3.4}.

\begin{thm}\label{thm15} 
Let $0 \le m<3\le n$. 
Then there is no algorithm which decides if the group of an $n$-knot
presentation is the group of an $m$-knot.
\end{thm}

\begin{proof} 
By Theorem~\ref{thm7} and Corollary~\ref{cor9} there is a recursive 
function $\psi$ associating to every finite group presentation $\Pi$ 
an $n$-knot presentation $\psi (\Pi)$ such that:
\begin{itemize}
\item[(i)] If $\Pi$ presents the trivial group then the group of $\psi(\Pi)$ 
is $\zed$, which is an $m$-knot group.

\item[(ii)] If $\Pi$ presents a non-trivial group then the group of 
$\psi(\Pi)$ is not a 2-knot group (and, therefore, not an $m$-knot group).
\end{itemize}

The theorem then follows from the undecidability of the triviality problem 
for group presentations.
\end{proof}

\section{The knotting problem}

Haken proved in \cite{Ha} that there is a procedure to decide if a 
given 1-knot is trivial. 
In this section we prove that if $n$ is such that there is a group in $\K_n$
with unsolvable word problem then it is impossible to find such a procedure 
for $n$-knots. 
Thus, if $n\ge 3$, there is no algorithm to decide if a given $n$-knot
is trivial; this has been proved by Nabutovsky and Weinberger \cite{NW}. 

Recall that we have given a recursive enumeration of all $n$-knot 
presentations $\P_1,\P_2,\ldots$.
A set $\{\P_i\}_{i\in S}$ of $n$-knot presentations is {\em recursive\/}
if and only if $S$ is recursive. 
Intuitively, $\{\P_i\}_{i\in S}$ is recursive if and only if 
there is an algorithm for 
determining whether or not a given knot presentation belongs to 
$\{\P\}_{i\in S}$.

\begin{thm}\label{thm16} 
Let $n$ be a natural number. If there is a group in $\K_n$ with unsolvable 
word problem then the set of presentations of $n$-knots which present the 
trivial knot is not recursive.
\end{thm}

\begin{proof} 
We may assume $n>1$ since the groups in $\K_1$ have solvable word problem 
(see \cite{W2}). 

We give first a sketch of the proof.

Suppose $U = |\langle \mu,x_1,\ldots,x_m : r_1,\ldots,r_p \rangle|$ is the 
group of the $n$-knot $(S^{n+2},\Gamma^n)$ where $U$ has unsolvable word 
problem and $\mu$ represents a meridian of $\Gamma^n$. 
Consider the knot $(S^{n+2},\Lambda^n)$ obtained by taking the connected 
sum of $(S^{n+2},\Gamma^n)$ with the trefoil spun $(n-1)$ times. 

Let $M^{n+2}$ be the manifold obtained by surgery on $(S^{n+2},\Lambda^n)$; 
the knot $\Lambda^n$ is replaced by a 1-sphere $S^1$. 
Let $\Sigma^n$ be a trivial $n$-sphere in $M^{n+2}-S^1$. 
Then, the fundamental group of $M^{n+2}$, which is isomorphic to that 
of $S^{n+2}-\Lambda^n$, is 
$$\pi_1 (M^{n+2}) = U *_\zed Y 
= \langle \mu,x_1,\ldots,x_m,y_1,y_2\ :\  
r_1,\ldots, r_p, y_1y_2y_1y_2^{-1} y_1^{-1} y_2^{-1}, \mu y_1^{-1}\rangle$$
where $Y$ is the trefoil group and the amalgamating subgroup $\zed$ is 
generated by $y_1$. 
Also 
$$\pi_1 (M^{n+2} -\Sigma^n) 
= \langle \sigma,\mu,x_1,\ldots,x_m,y_1,y_2\ :\ 
r_1,\ldots, r_p,y_1 y_2 y_1 y_2^{-1} y_1^{-1} y_2^{-1}, \mu y_1^{-1}\rangle$$
where $\mu$ represents $S^1$ and $\sigma$ a meridian of $\Sigma^n$.

To a word $w$ in $\mu,x_1,\ldots x_m$ associate a knot $(S_w^{n+2},\Sigma^n)$
where $S_w^{n+2}$ is obtained by surgery on $(M,\alpha)$, $\alpha$ being 
a 1-sphere in $M^{n+2}-\Sigma^n$ representing 
$\sigma^{-1} [w,y_2]^{-1} \sigma [w,y_2] \mu \in \pi_1 (M^{n+2} -\Sigma^n)$.
Notice that, as a 1-sphere in $M^{n+2}$, $\alpha$ represents 
$\mu\in \pi_1 (M^{n+2})$ and is therefore isotopic to $S^1$; this implies that 
$S_w^{n+2}$ is the $(n+2)$-sphere. 
Also, as we explain at the end of the proof, $(S_w^{n+2},\Sigma^n)$ is 
trivial if and only if $w=1$ in $U$. 

We show below that this function associating knots (or rather knot 
presentations) to words can be defined effectively. 
Hence if there were an algorithm deciding whether or not $n$-knots 
are trivial, there would be an algorithm which would solve the word 
problem in $U$.

We now proceed to give a more rigorous proof.

A simplicial complex $T$ with underlying polyhedron PL-homeomorphic to 
the manifold $M^{n+2}$ described above can be obtained by pasting together 
suitable simplicial complexes $E$ and $F$ with $|E|$ PL-homeomorphic to the 
exterior of $\Lambda^n$ and $|F| \approx S^1 \times D^{n+1}$. 
Also we may assume $E$ has subcomplexes $E_1$ and $E_2$ with $|E_1|$ 
PL-homeomorphic to the exterior of $\Gamma^n$, $|E_2|$ PL homeomorphic 
to the exterior of the spun trefoil and $|E_1|\cap |E_2| \approx S^1\times D^n$
with $\partial (|E_1| \cap |E_2|)$ containing a meridian of $\Lambda^n$.
We think of $E$ and $F$ as subcomplexes of $T$. 
We can assume $T$ contains a subcomplex $S$, disjoint from $F$, such that 
$|S| \approx S^n\times D^2$ and a core $\Sigma$ of $|S|$ bounds a PL 
$(n+1)$-disk in $|T| - |F|$.
Choose a vertex $*$ in $|E_1| \cap |E_2| \cap |F|$. 
One can find presentations $\langle\mu,x_1,\ldots,x_m : r_1,\ldots,r_p\rangle$, 
$\langle y_1,\ldots,y_k :s_1,\ldots,s_\ell\rangle$, 
$\langle \mu,x_1,\ldots,x_m,y_1,\ldots,y_k : r_1,\ldots,r_p,s_1,\ldots,
s_\ell,\mu y_1^{-1}\rangle$ and 
$\langle \sigma,\mu,x_1,\ldots,x_m,y_1,\ldots,y_k : 
r_1,\ldots,r_p, s_1,\ldots, s_\ell,\mu y_1^{-1}\rangle$ 
of $\pi_1 (|E_1|,*)$, $\pi_1 (|E_2|,*)$, $\pi_1(|T|,*)$ and 
$\pi_1 (|T| - |S|,*)$ respectively, by the usual method of taking a 
maximal tree in the 1-skeleton containing $*$, letting the generators be 
in a one-to-one correspondence with the remaining edges of the 1-skeleton 
and reading the relations from the 2-simplices. 
We can assume that a meridian of $\Lambda^n$ contained in 
$\partial (|E_1| \cap |E_2|)$ is represented by $\mu$ and by $y_1$, 
a meridian of $\Sigma$ is represented by $\sigma$, and $y_2$ represents an 
element of $\pi_1 (|E_2|,*)$ which does not commute with any non-trivial 
power of $y_1$. 
The inclusion-induced homomorphism $\pi_1 (|T| - |S|,*)\to \pi_1(|T|,*)$ 
sends $\sigma$ to 1, $\mu$ to $\mu$, $x_i$ to $x_i$ and $y_i$ to $y_i$. 

For each $r\ge1$, consider the $r$-th barycentric subdivision 
$(T^{(r)},S^{(r)})$ of the pair $(T,S)$. 
Every element of $\pi_1 (|T|- |S|,*)$ can be represented by an oriented 
PL 1-sphere containing $*$ which, by \cite[Corollary~1.6]{Hu} can be 
taken to be a subcomplex of $T^{(r)}$ for some $r$. 
We may assume that we know, for a given vertex $v$ of the subdivision 
$T^{(r)}$, the simplices of $T$ to which $v$ belongs. 
This enables one to give, for any $*$-based edge-loop 
(see \cite[sec.~6.3]{HW}) $\alpha$ in $T^{(r)}$, not meeting $|S|$, 
a $*$-based edge-loop in $T$ homotopic to it and, therefore, a word in 
$\sigma,\mu,x_1,\ldots,x_m,y_1,\ldots,y_k$, representing it; one can 
then recursively enumerate all words in $\sigma,\mu,x_1,\ldots,x_m,y_1,
\ldots,y_k$ representing $[\alpha]\in \pi_1 (|T| - |S|,*)$ since the 
words representing the trivial element of $|\langle\sigma,\mu,x_1,\ldots,x_m,
y_1,\ldots, y_k  : r_1,\ldots,r_p,s_1,\ldots,s_\ell\rangle|$ can be 
recursively enumerated.

Let $\Omega$ be a recursive enumeration of the triples $(r,C,u)$ such that 
\begin{itemize}
\item[(1)] $r$ is a positive integer,
\item[(2)] $C$ is an oriented 1-sphere in $|T| - |S|$ containing $*$, which 
is a subcomplex of $T^{(r)}$,
\item[(3)] $u$ is a word in $\sigma,\mu,x_1,\ldots,x_m,y_1,\ldots,y_k$ 
representing $[C]\in \pi_1 (|T| -|S|,*)$.
\end{itemize}

We now give a recursive function associating to every word $w$ in 
$\mu,x_1,\ldots,x_m$ a presentation $\P(w)$ of an $n$-knot. 
If $w$ is such a word, let $\Omega (j) = (r,C,u)$ be the triple with 
smallest $j$ such that $u = \sigma^{-1} [w,y_2]^{-1}$ 
$\sigma[w,y_2]\mu$ in the free group generated by 
$\sigma,\mu,x_1,\ldots,x_m,y_2$ and let 
$L = \{ \tau \in T^{(r+2)} : \tau\cap |C| = \emptyset\}$. 
Notice that $S^{(r+2)}$ is a subcomplex of $L$ so that, for every 
$q$, $S^{(r+2+q)}$, the $q$-th barycentric subdivision of $S^{(r+2)}$, is a 
subcomplex of $L^{(q)}$, the $q$-th barycentric subdivision of $L$. 
Recursively enumerate all triples $(\bA,\bD,\bB)$ such that $(\bA,\bD)$ 
is a presentation of an $n$-knot and $\bB$ is a subcomplex of 
$\overline{\bA-\bD}$; in this enumeration take the first triple 
$(\bA,\bD,\bB)$ such that (a realization of) $(\,\overline{\bA-\bD},\bB)$ 
is isomorphic to $(L^{(q)},S^{(r+2+q)})$ for some $q$, and define 
$\P(w) = (\bA,\bB)$.

To show that $\P(w)$ is well-defined we need only prove that in the last 
enumeration there is at least one triple $(\bA,\bD,\bB)$ such that 
$(\,\overline{\bA-\bD},\bB)$ is isomorphic to $(L^{(q)},S^{(r+2+q)})$ for 
some $q$. 
Since $\sigma=1$ in $\pi_1 (|T|,*)$, $[C]\in \pi_1(|T|,*)$ is represented by 
$\mu$ so $C$ is homotopic, and therefore isotopic, in $|T|$, 
to a core of $|F|$.  
Hence, $|L|$ is PL-homeomorphic to the knot exterior $|E|$. 
Let $D$ be a simplicial  complex such that $|D| \approx S^n\times D^2$.
Denote by $\partial D$ (resp.\ $\partial L$) the subcomplex of $D$ 
(resp.\ $L$) such that 
$|\partial D| = \partial |D|$ (resp.\ $|\partial L| = \partial |L|$) 
and let $f:\partial |D| \to \partial |L|$ be a PL homeomorphism such that 
$|D| \cup_f |L|$ is PL homeomorphic to $S^{n+2}$. 
By \cite[1.10, 1.6, 1.8 and 1.3(2)]{Hu} 
one may assume that $f:\partial D\to (\partial L)^{(q)}$ is a simplicial 
isomorphism for some $q$. 
Take an abstract complex pair $(\bD,\partial\bD)$ (resp.\ $(\bL,\bB)$) having 
$(D,\partial D)$ (resp.\ $(L^{(q)},S^{(r+2+q)})$ as a realization and 
let $\varphi :V(\partial \bD) \to V(\bL)$ correspond to $f$.
By changing the names of the vertices of $D$ if necessary, we can assume 
that $\varphi (v) =v$ for every $v\in V(\partial\bD)$ and that 
$\D\cap \bL = \partial \bD$.
If we now define $\bA = \bL\cup \bD$, then the triple $(\bA,\bD,\bB)$ has the 
required properties. 
Hence $\P(w)$ is well-defined. 

If $w=1$ in $U$ then $C$ is isotopic, in $\overline{|T|-|S|}$, to a core 
of $|F|$ and, therefore, there is a PL $(n+1)$-disk in $|T|$, bounded by 
a core of $|S|$, which does not intersect $C$. 
This implies that $\P(w)$ presents the trivial knot type. 

Now, the group $G_w$ of a knot in the knot type presented by $\P(w)$ is 
$$|\langle \sigma,\mu,x_1,\ldots,x_m,y_1,\ldots,y_k \ :\ 
r_1 = 1,\ldots, r_p = 1,\  s_1 = 1,\ldots,s_\ell = 1,\  
\sigma^{-1} [w,y_2]\sigma = [w,y_2]\mu\rangle |\ .$$

Now, $[w,y_2]\mu$ has infinite order in 
$|\langle\mu,x_1,\ldots,x_m,y_1,\ldots,y_k
: r_1,\ldots,r_p, s_1,\ldots,s_\ell,\mu y_1^{-1} \rangle |$ 
\break ($= \pi_1 (S^{n+2} -\Lambda^n)$). 
If $w$ does not represent the trivial element of 
$|\langle\mu ,x_1,\ldots,x_m : r_1 ,\ldots,r_p\rangle|$ 
then also $[w,y_2]$ has infinite 
order in $\pi_1 (S^{n+2}-\Lambda^n)$ (here one uses that $[y_1^n,y_2]\ne1$ 
for any $n\ne 0$) and therefore $G_w$ is an HNN extension of 
$\pi_1 (S^{n+2}-\Lambda^n)$.

Thus if $w \ne1$ in $|\langle  \mu,x_1,\ldots,x_m : r_1,\ldots,r_p\rangle|$ 
then $\P(w)$ presents a non-trivial knot type.

Hence, if the set of presentations of $n$-knots defining the trivial knot
type were recursive, the word problem in $U$ would be solvable, which 
is not the case.
\end{proof}

Since $\K_n = \K_3$ for $n\ge3$ and $\K_3$ contains groups with unsolvable 
word problem by Corollary~\ref{cor:wordprob}
one has the following corollary (cf.\ \cite{NW}).

\begin{cor}[Nabutovsky-Weinberger]\label{cor:NW} 
If $n\ge3$ then the set of presentations of $n$-knots which present 
the trivial knot is not recursive.
\end{cor}

\subsection*{Remarks.}$\quad$

\begin{itemize}
\item[(1)] If in the proof of Theorem~\ref{thm16} one can take $U$ 
torsion-free (as one may if $n\ge3$) a slightly simpler proof can be given: 
there is no need to take the connected sum with a spun trefoil and instead 
of the word $\sigma^{-1} [w,y_2]^{-1}\sigma [w,y_2]\mu$ one can take 
$\sigma^{-1} w^{-1} \sigma w\mu$.

\item[(2)] If $n\ge3$ then any property enjoyed by the trivial $n$-knot 
but not by any of the knots $\P(w)$ of the proof of Theorem~\ref{thm16} 
with $w\ne1$ is not algorithmically recognizable. 
Among these are:
\begin{itemize}
\item[(i)] Being a fibered knot.
\item[(ii)] Having a group with finitely generated (or presented) commutator 
subgroup.
\item[(iii)] Having a group with solvable word problem.
\item[(iv)] Having a torsion-free group (here take $U$ with torsion).
\item[(v)] If $H$ is a non-trivial group with $H\not\cong \zed$, having a group 
not containing $H$ as a subgroup (here take $U$ containing $H$).
\end{itemize}
\end{itemize}

In conclusion here are some questions.
\begin{itemize}
\item[(1)] {\em Is there a 2-knot group with unsolvable word problem?
Conjecture: Yes.}
\item[(2)] {\em Does each finitely presented group embed in a 2-knot 
group? Conjecture: Yes.} 
\item[(3)] {\em If $g$ is a non-negative integer, is there an algorithm 
to decide whether or not a given locally flat PL-embedded surface 
of genus~$g$ in $S^4$ is unknotted? 
Conjecture: No for any value of $g$.}
\end{itemize}


\vskip.2in 

{\small\parindent=0pt\parskip=0pt
Francisco Gonz\'alez-Acu\~na\par
Instituto de Matem\'aticas\par
UNAM\par
Ciudad Universitaria 04510\par
M\'exico D.F., Mexico\par
$\quad$ and\par
CIMAT A.P. 402\par
Guanajuato, Gto., Mexico\par
ficomx@yahoo.com.mx\par
\bigskip

Cameron McA. Gordon\par
Department of Mathematics\par
The University of Texas at Austin\par
1 University Station -- C1200\par
Austin, TX 78712-0257 U.S.A.\par
gordon@math.utexas.edu\par
\bigskip

Jonathan Simon\par
Department of Mathematics\par
University of Iowa\par
Iowa City, IA 52242 U.S.A.\par
jsimon@math.uiowa.edu\par

}

\end{document}